\documentclass[11pt]{amsproc}
\author{Paul Pollack}
\address{Department of Mathematics\\ University of Georgia\\ Athens, GA 30602}
\email{pollack@uga.edu}
\author{Carl Pomerance}
\address{Department of Mathematics, Santa Clara University, Santa Clara, CA 95053 ({\rm current})\\ Department of Mathematics\\ Dartmouth College\\ Hanover, NH 03755}
\email{carl.pomerance@dartmouth.edu}

\title{Phi, Primorials, and Poisson}
\usepackage{amsmath,amssymb,amsthm,microtype,geometry,mathscinet}
\usepackage[utf8]{inputenc}

\DeclareMathAlphabet{\curly}{U}{rsfs}{m}{n}
\newtheorem{thm}{Theorem}
\newtheorem{cor}[thm]{Corollary}
\newtheorem{prop}[thm]{Proposition}
\newtheorem{lem}[thm]{Lemma}

\theoremstyle{remark}
\newtheorem*{rmk}{Remark}
\begin{document}
\renewcommand{\labelenumi}{(\roman{enumi})}
\def\Ll{\mathcal{L}}
\def\N{\mathbb{N}}
\def\Q{\mathbb{Q}}
\newcommand\rad{\mathrm{rad}}
\def\Z{\mathbb{Z}}
\def\R{\mathbb{R}}
\def\C{\mathbb{C}}
\def\Pp{\mathcal{P}}
\newcommand\Li{\mathrm{Li}}
\begin{abstract}
The primorial $p\#$ of a prime $p$ is the product of all primes $q\le p$.
Let ${\rm pr}(n)$ denote the largest prime $p$ with $p\#\mid\phi(n)$, where
$\phi$ is Euler's totient function.
We show that the normal order of ${\rm pr}(n)$ is $\log\log n/\log\log\log n$.  That is,
${\rm pr}(n)\sim\log\log n/\log\log\log n$
as $n\to\infty$ on a set of integers of asymptotic density 1.  In fact we show there is an asymptotic secondary term and, on a tertiary level,
there is an asymptotic Poisson distribution.  We also show an analogous result for the
largest integer $k$ with $k!\mid \phi(n)$.
%We study the largest primorial dividing $\phi(n)$, where $\phi(n)$ is Euler's $\phi$-function. Let
%\[ A(x) = \frac{\log\log{x}}{\log\log\log{x}} + 3 \frac{\log\log{x}\cdot \log\log\log\log{x}}{(\log\log\log{x})^2} \]
%and
%\[ B(x) = \frac{\log\log{x}}{(\log\log\log{x})^2}.\]
%We show that whenever $\Lambda(x)\to\infty$ as $x\to\infty$, (asymptotically) almost all $n \le x$ have $\phi(n)$ divisible by all primes up to $A(x) - \Lambda(x) B(x)$, while almost no $n\le x$ have $\phi(n)$ divisible by all primes up to $A(x)+\Lambda(x) B(x)$. This is deduced as a corollary to a more precise theorem: For each fixed $\lambda > 0$, the statistic
%\[ \#\{\text{primes }p \le A(x) + (\log \lambda) B(x): p \nmid \phi(n)\},  \]
%computed over $n\le x$, is asymptotically Poisson-distributed with parameter $\lambda$.
\end{abstract}

\maketitle

\section{Introduction}
Euler's \emph{totient function} $\phi(n)$ may be defined as the number of units in the residue ring $\Z/n\Z$, or equivalently via the formula
\[ \phi(n) = n \prod_{\ell \mid n}\left(1-\frac{1}{\ell}\right), \]
where the product on $\ell$ is over the distinct primes dividing $n$. Our starting point in this article is the following remarkable property of $\phi$: For every fixed prime number $p$, almost every value of $\phi(n)$ is divisible by $p$. Here `almost every' means that, as $x\to\infty$, all but $o(x)$ values of $n\le x$ are such that $p \mid \phi(n)$.

How could this possibly be the case? A small piece of probabilistic reasoning dispells the mystery: Observe that $\phi(n)$ is divisible by $p$ whenever $n$ is divisible by a prime $\ell\equiv 1\pmod{p}$. Those primes $\ell$ make up a  positive proportion of all primes, namely $1$ in $p-1$, by the prime number theorem for arithmetic progressions. Almost all numbers $n\le x$ have $\approx \log\log{x}$ distinct prime factors (a classical result of Hardy and Ramanujan), and it should be unusual for these many prime factors to all avoid the residue class $1$ mod $p$. This argument is merely heuristic, but can be made rigorous by sieve methods or by analytic methods going back to Landau (further developed  by Selberg and Delange). An early reference for this fact about $\phi$ is \cite{AE44}; that paper treats $\sigma(n)$ (the sum-of-divisors function) rather than $\phi(n)$, but the proof is almost the same.

It follows that there are functions $y=y(x)$, tending monotonically to infinity, with the property that all but $o(x)$ values of $n \le x$ are divisible by $\prod_{p \le y} p$, as $x\to\infty$. It is implicit in the arguments of Erd\H{o}s in \cite{erdos48} (see also \cite{erdos61}) that $y=(\log\log{x})^{1-\epsilon}$ is admissible, for any fixed $\epsilon \in (0,1)$. Erd\H{o}s's reasoning is developed in \cite{EGPS90} and \cite{LP02}, where it is shown that we may take $y = c \log\log{x}/\log\log\log{x}$ for some positive constant $c$. Among other things, our main result yields a very precise determination of the allowable values of $y$.

We need a bit of set-up to state our main theorem. With $\log_k$ denoting the $k$-fold iterated logarithm, we set
\[ A(x) = \frac{\log_2{x}}{\log_3{x}} + 3\frac{\log_2 x \cdot \log_4 x}{(\log_3 x)^2}, \quad B(x) = \frac{\log_2{x}}{(\log_3 x)^2}. \]
For $n \le x$ and $\Lambda$ real, we set
\[ f(n,\Lambda) = \#\{\text{primes }p \le A(x) + \Lambda \cdot B(x): p \nmid \phi(n)\}.\]

\begin{thm}\label{thm:main} Fix $\lambda > 0$. Then $f(n,\log\lambda)$, as a statistic on integers $n\le x$, is asymptotically Poisson distributed with parameter $\lambda$. That is, for each fixed nonnegative integer $k$, the proportion of $n\le x$ with
\[ \#\{\text{primes }p \le A(x) + (\log\lambda) B(x), p \nmid \phi(n)\} = k \]
tends to $\displaystyle e^{-\lambda} \frac{\lambda^k}{k!}$, as $x\to\infty$.
\end{thm}

Taking $k=0$ in Theorem \ref{thm:main}, we deduce:
\setlength{\fboxsep}{7pt}
\begin{equation}\label{eq:probability}
\fbox{\begin{minipage}{32em}
The limiting proportion of $n\le x$ with $\phi(n)$ divisible by all primes up to $A(x) + (\log \lambda) B(x)$ is $\mathrm{e}^{-\lambda}$.
\end{minipage}}
\end{equation}

This has the following immediate consequence.
\begin{cor}\label{cor:almostall} Whenever $\Lambda=\Lambda(x) \to \infty$ as $x\to\infty$, almost all $n \le x$ have $\phi(n)$ divisible by all primes up to $A(x) - \Lambda(x) B(x)$, while almost no $n\le x$ have $\phi(n)$ divisible by all primes up to $A(x) + \Lambda(x) B(x)$.
\end{cor}

As the astute reader may have noticed, the argument sketched at the start of the introduction works equally well to show that $\phi(n)$ is almost always divisible by any fixed integer $m$ (not necessarily prime!). This point of view suggests studying the largest factorial dividing $\phi(n)$.  In \S\ref{sec:factorial} we establish the natural factorial analogues of \eqref{eq:probability} and Corollary \ref{cor:almostall}.

In \S\ref{sec:lambda} we consider the primorial and factorial problems for Carmichael's universal exponent function $\lambda(n)$.
Finally, in \S\ref{sec:other} we raise the related questions where instead of asking what occurs for almost all $n$,
we ask what occurs for almost all $\phi$-values, or $\lambda$-values.

\subsection*{Notation} Throughout, we reserve the letters $\ell$ and $p$ for primes.
We use the notation $v_p(n)$ to denote the largest integer $v$ with $p^v\mid n$.

\section{Proof of Theorem \ref{thm:main}}\label{sec:mainproof}

\begin{lem}\label{lem:sieve} Let $\Pp$ be a set of primes, let $x\ge 1$, and let $S = \sum_{\ell \in \mathcal{P},~\ell \le x} \frac{1}{\ell}$.
Uniformly for all choices of $\Pp$,
the proportion of $n\le x$ free of prime factors from $\Pp$ is
$\ll \exp(-S)$.
\end{lem}
\begin{proof}
This result follows from Brun's sieve, see \cite[Theorem 2.2]{HR74}.
\end{proof}

\begin{lem}\label{lem:reciprocals} Let $m$ be a positive integer, and let $x\ge 3$. Put
\[ S(x;m) = \sum_{\substack{\ell \le x \\ \ell \,\equiv\,1\kern-3pt\!\!\!\pmod{m}}} \frac{1}{\ell}. \]
Then $$S(x;m) = \frac{\log_2{x}}{\phi(m)} + O\left(\frac{\log{(2m)}}{\phi(m)}\right).$$
\end{lem}
\begin{proof} See Remark 1 of \cite{Pomerance77} or the Lemma on p.\ 699 of \cite{norton76}.
\end{proof}

It will be convenient for the proof of Theorem \ref{thm:main} to work not with $f(n,\Lambda)$ but with a variant function that seems less natural but is more amenable to analysis. Put $$ \tilde{\phi}(n) = \prod_{\substack{\ell\mid n \\ \ell \le x^{1/\log_3{x}}}}(\ell-1),$$
let $$A_0(x) = \log_2{x}/\log_3{x},$$ and define
\[ g(n,\Lambda) = \#\{p: A_0(x) < p \le A(x) + \Lambda \cdot B(x), \text{ and } p \nmid \tilde{\phi}(n)\}.\]
The next lemma assures us that for the density results we aim at, there is no difference dealing with $g$ versus $f$.

\begin{lem}\label{lem:fvsg}Fix a real number $\Lambda$. The proportion of $n \le x$ with $f(n,\Lambda) \ne g(n,\Lambda)$ tends to $0$, as $x\to\infty$.
\end{lem}

\begin{proof} Suppose that $f(n,\Lambda) \ne g(n,\Lambda)$. Then either there is a prime $p$ counted by $f$ and not by $g$, or vice versa.

In the first case, we must have $p \le A_0(x)$. Since $p\nmid \phi(n)$, there is no prime $\ell \equiv 1\pmod{p}$ for which $\ell \mid n$. By Lemmas \ref{lem:sieve} and \ref{lem:reciprocals}, the proportion of  $n \le x$ satisfying this latter condition is
\[ \ll  \exp(-S(x;p))\le \exp\left(-\frac{\log_2x}{p-1}+O(1)\right)\ll \frac1{\log_2x}. \]
Summing on $p \le A_0(x)$, we find that the proportion of $n\le x$ occurring in this first case is $O(1/(\log_3{x})^2)$, and so is $o(1)$.

In the second case, $p > A_0(x)$ and $p \mid \phi(n)$, but $p \nmid \tilde{\phi}(n)$. Thus, either
\begin{enumerate}
\item[(i)] $p^2\mid n$, \emph{or}
\item[(ii)] there is a prime $\ell \mid n$, $\ell\equiv 1\pmod{p}$ with $\ell > x^{1/\log_3{x}}$.
\end{enumerate}
The proportion of $n\le x$ for which (i) can occur (for some $p$) is $\ll \sum_{p > A_0(x)}\frac{1}{p^2}$, and so is $o(1)$. The proportion of $n\le x$ for which (ii) can occur is
\[ \ll \sum_{A_0(x) < p \le A(x) + \Lambda \cdot B(x)} \sum_{\substack{\ell \,\equiv\,1\kern-3pt\pmod{p} \\ x^{1/\log_3{x}} < \ell \le x}} \frac{1}{\ell}. \]
By Brun--Titchmarsh and partial summation, the inner sum on $\ell$ is $O((\log_4{x})/p)$, making the last display (for large $x$)
\begin{multline*} \ll \log_4{x} \sum_{A_0(x) < p \le A(x) + \Lambda \cdot B(x)}\frac{1}{p} \ll  \frac{\log_4{x}}{A_0(x)} \cdot \#\{p: A_0(x) < p \le A(x) + \Lambda \cdot B(x)\} \\ \le \frac{\log_4{x}}{A_0(x)} \cdot \#\{p: A_0(x) < p \le 2A_0(x)\}\
\ll \frac{\log_4{x}}{A_0(x)} \cdot \frac{A_0(x)}{\log A_0(x)} \ll \frac{\log_4{x}}{\log_3{x}}.\end{multline*}
 Thus, the proportion of $n\le x$ as in (ii) is also $o(1)$.
\end{proof}

In view of Lemma \ref{lem:fvsg}, to prove Theorem \ref{thm:main} it suffices to show that $g(n,\log\lambda)$ is asymptotically Poisson distributed with parameter $\lambda$. The Poisson distribution of parameter $\lambda$, which we will denote by $\texttt{Po}(\lambda)$, is determined by its moments (see, for instance, Theorem 30.1 on p.\ 388 of \cite{billingsley95}, along with Example 21.4 on p.\ 279 there). It is equivalent, but somewhat simpler here, to work with factorial moments instead of moments. The $r$th factorial moment of $\texttt{Po}(\lambda)$ is $\lambda^r$, and so by the Fr\'echet-Shohat moment theorem (see \cite[Theorem 28, p. 81]{galambos95}) it is enough to prove that \[ \lim_{x\to\infty} \frac{1}{x} \sum_{n\le x} g(n,\log\lambda) (g(n,\log\lambda)-1) \cdots (g(n,\log\lambda)-(r-1)) = \lambda^r \]
for each fixed $r=1,2,3,\dots$.

By induction on $r$, or by first recognizing the falling factorial as the numerator of a binomial coefficient, we see that
\begin{multline}\label{eq:gexpansion} \frac{1}{x}\sum_{n \le x} g(n,\log\lambda) (g(n,\log\lambda)-1) \cdots (g(n,\log\lambda)-(r-1)) \\ =
\frac{1}{x}\sum_{\substack{A_0(x) < p_1, \dots, p_r \le A(x) + (\log \lambda) B(x) \\ p_1,\dots,p_r \text{ distinct}}} \#\{n\le x: \gcd(\tilde{\phi}(n),p_1\cdots p_r)=1\}. \end{multline}
Fix distinct primes $p_1,\dots,p_r$ as in the sum.
Putting $\Ll = \{\ell \le x^{1/\log_3 x}: \ell \equiv 1\pmod{p_i}\text{ for some $i$}\}$, the right-hand summand in \eqref{eq:gexpansion} counts those $n\le x$ not divisible by any prime $\ell \in \Ll$. By the fundamental lemma of the sieve\footnote{Specifically, we use the following consequence of Theorem 2.5 in \cite{HR74}: If $\Ll$ is any set of primes not exceeding $x^{1/\log_3 x}$, then the number of $n\le x$ not divisible by any member of $\Ll$ is $\sim x\prod_{\ell \in \Ll} (1-1/\ell)$, as $x\to\infty$, uniformly in $\Ll$.}, this count is $\sim x \prod_{\substack{\ell \in\Ll}} \left(1-\frac{1}{\ell}\right)$ as $x\to\infty$, where the asymptotic holds uniformly in $p_1,\dots,p_r$. Since each $\ell > A_0(x)$, this is in turn $\sim x\exp(-T)$, where $T = \sum_{\ell \in \Ll} \frac{1}{\ell}$.
Now
\[ T = \sum_{i=1}^{r} \sum_{\substack{\ell \le x^{1/\log_3 x} \\ \ell \,\equiv\,1\kern-3pt\pmod{p_i}}}\frac{1}{\ell} + O\bigg(\max_{i,j} \sum_{\substack{\ell \le x^{1/\log_3 x} \\ \ell \,\equiv\,1\kern-3pt\pmod{p_i p_j}}} \frac{1}{\ell} \bigg). \]
(We suppress the dependence of the implied constant on $r$, which is fixed.) Since each $p_i p_j > A_0(x)^2 > (\log_2{x})^{1.9}$, Lemma \ref{lem:reciprocals} implies that the $O$-term here is $o(1)$, as $x\to\infty$. Thus, $\exp(-T) \sim \prod_{i=1}^{r} \exp\big(-\sum_{\substack{\ell \le x^{1/\log_3{x}} \\ \ell \,\equiv\,1\kern-3pt\pmod{p_i}}}\frac{1}{\ell}\big)$,
and
\[ \frac{1}{x}\#\{n\le x: \gcd(\tilde{\phi}(n),p_1\cdots p_r)=1\} \sim \prod_{i=1}^{r} \exp\bigg(-\sum_{\substack{\ell \le x^{1/\log_3{x}} \\ \ell \,\equiv\,1\kern-3pt\pmod{p_i}}}\frac{1}{\ell}\bigg).\]
By Lemma \ref{lem:reciprocals}, the remaining sum on $\ell$  is $ \frac{1}{p_i-1}\log_2{(x^{1/\log_3 x})}+o(1) = \frac{1}{p_i}\log_2{x} + o(1)$, so that the RHS displayed above is $\sim \prod_{i=1}^{r} \exp(-\log_2{x}/p_i)$. All of our asymptotic results hold uniformly in $p_1,\dots,p_r$, and so summing on $p_1,\dots,p_r$ yields
\begin{multline}\label{eq:gexpansion2}\frac{1}{x}\sum_{n \le x} g(n,\log\lambda) (g(n,\log\lambda)-1) \cdots (g(n,\log\lambda)-(r-1)) \\
\sim \sum_{\substack{A_0(x) < p_1, \dots, p_r \le A(x) + (\log \lambda) B(x) \\ p_1,\dots,p_r \text{ distinct}}} \prod_{i=1}^{r} \exp\bigg(-\frac{\log_2{x}}{p_i}\bigg). \end{multline}

We briefly digress to study the effect of removing the distinctness condition on the $p_i$ from this last expression. The resulting sum is the $r$th power of
\[ \sum_{A_0(x) < p \le A(x) + (\log \lambda) B(x)} \exp\bigg(-\frac{\log_2{x}}{p}\bigg) = \int_{A_0(x)}^{A(x) + (\log \lambda) B(x)} \exp\left(-\frac{\log_2 x}{u}\right)\, \mathrm{d}\pi(u). \]
 Write $\pi(u) = \int_2^{u} \frac{\mathrm{d}{u}}{\log{u}} + E(u)$,
 so that $\mathrm{d}\pi(u) = \frac{\mathrm{d}u}{\log{u}} + \mathrm{d}E(u)$.
 Using that $E(u) \ll_{K} u/(\log{u})^{K}$ for every fixed $K$ (a strong form of the prime number theorem), a straightforward computaton shows that the integral above, with $\mathrm{d}\pi(u)$ replaced by $\mathrm{d}E(u)$, is $o(1)$. Turning to the remaining piece of  integral, we see that
\[ \int_{A_0(x)}^{A(x) + (\log \lambda) B(x)} \exp\left(-\frac{\log_2 x}{u}\right)\, \frac{\mathrm{d}u}{\log{u}} \sim \frac{1}{\log_3 x}\int_{A_0(x)}^{A(x) + (\log \lambda) B(x)} \exp\left(-\frac{\log_2 x}{u}\right)\, \mathrm{d}u. \]
Make the change of variables $u=A_0(x)(1+z)$. Then
\[ \frac{1}{\log_3{x}} \int_{A_0(x)}^{A(x) + (\log \lambda) B(x)}\exp\left(-\frac{\log_2 x}{u}\right)\, \mathrm{d}u = \frac{A_0(x)}{\log_3{x}}  \int_{0}^{\frac{3\log_4{x} + \log\lambda}{\log_3{x}}} \exp\left(-\frac{\log_3{x}}{1+z}\right)\, \mathrm{d}z.\]
Now $\frac{\log_3{x}}{1+z} = (\log_3{x}) (1-z) + o(1) = \log_3{x} - z\log_3{x} + o(1)$, uniformly for $0 \le z \le \frac{3 \log_4{x} + \log\lambda}{\log_3{x}}$. Therefore,
\begin{align*} \frac{A_0(x)}{\log_3{x}}  \int_{0}^{\frac{3\log_4{x} + \log\lambda}{\log_3{x}}} \exp\left(-\frac{\log_3{x}}{1+z}\right)\, \mathrm{d}z &\sim \frac{A_0(x)}{\log_2{x}\log_3{x}} \int_{0}^{\frac{3\log_4{x} + \log\lambda}{\log_3{x}}} \exp(z\log_3{x})\, \mathrm{d}z\\
&\sim \frac{A_0(x)}{\log_2 x(\log_3{x})^2} \exp(3\log_4 x+\log\lambda) = \lambda.\end{align*}
Collecting all of the estimates of this paragraph, we conclude that as $x\to\infty$,
\[ \sum_{A_0(x) < p \le A(x) + (\log \lambda) B(x)} \exp\bigg(-\frac{\log_2{x}}{p}\bigg) \to \lambda. \]

When $r=1$, the work of the last paragraph establishes that the left-hand side of \eqref{eq:gexpansion2} converges to $\lambda^r$. We also see that the same convergence assertion will follow for $r>1$ provided that
\[ \sum_{\substack{A_0(x) < p_1, \dots, p_r \le A(x) + (\log \lambda) B(x) \\ \text{some $p_i=p_j$ with $i\ne j$}}} \prod_{i=1}^{r} \exp\bigg(-\frac{\log_2{x}}{p_i}\bigg) = o(1). \]
If some $p_i=p_j$, then reordering  the $p_i$, we can force $p_1=p_2$. Thus, the above left-hand side is
\begin{multline*}
 \ll \sum_{ p, p_3, \dots, p_r} \exp\left(-2\frac{\log_2{x}}{p}\right) \prod_{i=3}^{r}\exp\left(-\frac{\log_2{x}}{p_i}\right)  \\
= \left(\sum_{p} \exp\left(-2\frac{\log_2{x}}{p}\right)\right) \prod_{i=3}^{r} \left(\sum_{p_i} \exp\left(-\frac{\log_2{x}}{p_i}\right)\right) \ll \sum_{p} \exp\left(-2\frac{\log_2{x}}{p}\right),
\end{multline*}
where, as above, $p$ and the $p_i$ range over $(A_0(x),A(x)+(\log\lambda)B(x)]$. The final sum on $p$ is $o(1)$, since each summand is $\ll (\log_2{x})^{-1.9}$ (say), and there are crudely $O(\log_2{x})$ summands. This completes the proof of Theorem \ref{thm:main}.

\begin{rmk} One could ask not only for every prime up to a certain height to appear in $\phi(n)$, but for those primes to appear to at least the $r$th power, for one's favorite fixed positive integer $r$. The above analysis can be adapted to prove an analogue of Theorem~\ref{thm:main} in this generalized setting. Define
\[ A_r(x) = \frac{\log_2{x}}{\log_3{x}} + (4-r)\frac{\log_2{x}\cdot \log_4{x}}{(\log_3{x})^2}. \]
For integers $n\le x$ and real $\Lambda$, set
\[ f_r(n,\Lambda) = \#\left\{p \le A_r(x) + \Lambda \cdot B(x): p^r\nmid \phi(n)\right\}. \]
In analogy with Theorem \ref{thm:main} (the case $r=1$), we can show that for each fixed $\lambda > 0$, the quantity $f_r(n,\log\{(r-1)!\lambda\})$, considered on the integers $n\le x$, is asymptotically Poisson distributed with parameter $\lambda$. The broad outline of the proof is the same as before; roughly speaking, the numbers with no small prime factors from the progression $1\bmod{p}$ have their role replaced by those with at most $r-1$ such prime factors.  This thought is
made more explicit in the next section.
\end{rmk}

\section{From primorials to factorials}\label{sec:factorial}
In this section we prove the following analogue of \eqref{eq:probability}.
\begin{equation}\label{eq:probability2}
\fbox{\begin{minipage}{32em} Fix $\lambda > 0$.
The limiting proportion of $n\le x$ with $\phi(n)$ divisible by $[y]!$, where $y=A(x) + (\log \lambda) B(x)$, is $\mathrm{e}^{-\lambda}$.
\end{minipage}}
\end{equation}
Note that the obvious counterpart of Corollary \ref{cor:almostall} follows as an immediate consequence.

Clearly, if $\lfloor y\rfloor! \mid \phi(n)$, then $n$ is divisible by the product of all primes up to $y$. We will show that if $n \le x$ and $\phi(n)$ is divisible by the product of all primes up to $y$, then apart from $o(x)$ exceptions, $\phi(n)$ is divisible by $\lfloor y\rfloor!$. Thus, \eqref{eq:probability2} follows from \eqref{eq:probability}.

For this, we appeal to the following generalization of Lemma \ref{lem:sieve}, due essentially to  Hal\'asz \cite{halasz72}.

\begin{prop}\label{prop:halasz} Let $\Pp$ be a set of primes, let $x\ge 1$, and let $S = \sum_{\ell \in \mathcal{P},~\ell \le x} \frac{1}{\ell}$.
Suppose that $0 < \delta < 2$. Then for each integer $m$ with $0 \le m \le (2-\delta) S$, the proportion of $n\le x$ with exactly $m$ distinct prime factors from $\Pp$ is
\[ \ll_{\delta}  \exp(-S) \frac{S^{m}}{m!}. \]
\end{prop}

\noindent Actually, Hal\'asz counts prime factors with multiplicity, rather than distinct prime factors. The modifications necessary to establish the proposition as we have stated it are described by Norton on p.\ 688 of \cite{norton76}.

We suppose now that $n \le x$, that $\phi(n)$ is divisible by all primes up to $y$ (with $y=y(x)$ as in \eqref{eq:probability2}), but that $\lfloor y\rfloor! \nmid \phi(n)$. Then we can find a prime $p\le y$ with
\[ 1 \le v_p(\phi(n)) < v_p(\lfloor y\rfloor!). \]
Since $2\le v_p(\lfloor y\rfloor!) = \sum_{r \ge 1} \lfloor y/p^r\rfloor < \frac{y}{p-1}$, we have $p-1 < y/2$. Choose the integer $k\ge 2$ with
\begin{equation}\label{eq:krange} \frac{y}{k+1} \le p-1 < \frac{y}{k}. \end{equation}
Then $v_p(\phi(n)) \le k$, and so $n$ is divisible by at most $k$ primes from $\Pp =\{\ell\equiv 1\pmod{p}\}$. Put $S = \sum_{\ell \le x,~ \ell \,\equiv\,1\kern-3pt\pmod{p}} \frac{1}{\ell}$, so that
\begin{align*} S &= \frac{1}{p-1}\log_2{x} + O(1) \ge k \frac{\log_2{x}}{y} + O(1) > \frac{9}{10}k\log_3{x} + O(1) > 2k
\end{align*}
for large $x$. By Proposition \ref{prop:halasz}, the proportion of $n$ divisible by at most $k$ primes from $\Pp$ is
\[ \ll \exp(-S) \sum_{j=0}^{k} \frac{S^j}{j!} \ll \exp(-S)\frac{S^k}{k!} \ll \exp(-0.9k\log_3{x}) \frac{S^k}{k!}. \]
Using $S \le \frac{\log_2{x}}{p-1} + O(\frac{\log(2p)}{p-1})\ll \frac{\log_2{x}}{p-1} \ll k\log_3{x}$ (by Lemma \ref{lem:reciprocals}) and $k! \ge (k/\mathrm{e})^k$, we find that the last displayed expression is
\[ \ll \left(C \frac{\log_3{x}}{(\log_2{x})^{0.9}}\right)^k, \]
where $C$ is a certain absolute constant.

It remains to sum on the $p$'s corresponding to a given $k$, and then to sum on $k$. To each $k\ge 2$, there are (very crudely) $\ll y/\log{y} \ll \log_2{x}/(\log_3{x})^2$ primes $p$ in the range determined by \eqref{eq:krange}. Hence, the proportion of $n$ with $\phi(n)$ divisible by all primes up to $y$, but not by $\lfloor y\rfloor!$, is
\[ \ll \sum_{k\ge 2} \frac{\log_2{x}}{(\log_3{x})^2} \left(C \frac{\log_3{x}}{(\log_2{x})^{0.9}}\right)^k \ll (\log_2{x})^{-0.8}, \]
which tends to $0$ as desired.

\section{Carmichael's function}\label{sec:lambda}
One might ask about analogues of Theorem \ref{thm:main} and the factorial problem of \S\ref{sec:factorial}
for other number theoretic functions similar to $\phi$.  As one might expect, we have the same theorems for
the sum-of-divisors function $\sigma$, since the only complications are nontrivial prime powers, and for
almost all $n$, nontrivial prime power divisors are small.

We now ask about Carmichael's function $\lambda(n)$.  It is the
 order of the largest cyclic subgroup of $({\mathbb Z}/n{\mathbb Z})^*$; namely, the
exponent of the unit group of ${\mathbb Z}/n{\mathbb Z}$.  Carmichael's function is closely related
to Euler's function $\phi$.  In fact, from the theorem on the primitive root and from the Chinese Remainder Theorem,
we have that
\begin{align*}
\lambda(p^a)&=\phi(p^a)\hbox{ for $p>2$ or $p^a<8$},\\
\lambda(2^a)&=\frac12\phi(2^a)\hbox{ for $a\ge3$},\\
\lambda(mn)&={\rm lcm}[\lambda(m),\lambda(n)]\hbox{ when $\gcd(m,n)=1$}.
\end{align*}
It is immediate that $p\mid\phi(n)$ if and only if $p\mid\lambda(n)$, so that we have the analogue of Theorem~\ref{thm:main}
for Carmichael's function.\footnote{It is unfortunate that mathematics uses the same symbol ``$\lambda$" for a
Poisson variable as for Carmichael's function; we trust there will be no confusion.}

The situation though for factorials is markedly different.  Let $k_\lambda(n)$ denote the largest integer $k$ with
$k!\mid\lambda(n)$.
\begin{lem}
\label{lem:2power}
Let $\xi(n)\to\infty$ arbitrarily slowly.
There is a set of integers $S$ of asymptotic density $1$ such that for $n\in S$,
\begin{enumerate}
\item
$\frac1{\xi(n)}\log_2n\le \max\{2^{v_2(p-1)}:p\mid n\}\le\xi(n)\log_2n$,
\item
$v_2(\lambda(n))=\max\{v_2(p-1):p\mid n\}$,
\item
 $k_\lambda(n)$ is the largest integer $k$ with
$v_2(k!)\le v_2(\lambda(n))$.
\end{enumerate}\end{lem}
\begin{proof}
We may assume that $\xi(x)\le\log_3x$.  Let $2^m$ be the least power of 2 exceeding $(\log_2x)/\xi(x)$ and let
$2^M$ be the largest power of 2 not exceeding $\xi(x)\log_2x$.  It follows from Lemmas \ref{lem:sieve},  \ref{lem:reciprocals}
that but for $o(x)$ choices for integers $n\le x$ we have a prime $p\mid n$ with $p\equiv1\pmod{2^m}$.  Further, it
follows from Lemma \ref{lem:reciprocals} that the proportion of integers $n\le x$ divisible by a
 prime $p\equiv1\pmod{2^M}$ is $\ll( \log_2x)/2^M=o(x)$.  Thus, we have (i).  The only way that (ii) would not
hold is if the 2-power in $\lambda(n)$ is $\lambda(2^{v_2(n)})$.  If also $n$ satisfies (i), as we may assume,
and $n\le x$, this would imply that $2^{v_2(n)}>2^m> (\log_2x)/\xi(x)$.  The number of such $n$ is $O(x\xi(x)/\log_2x)=o(x)$
as $x\to\infty$.  Thus, we have (ii).

Using (i) and (ii) we have but for $o(x)$ choices of $n\le x$ that
\begin{equation}
\label{eq:v2lambda}
v_2(\lambda(n))=\frac{\log_3x}{\log2}+O(\log\xi(x)).
\end{equation}
It is clear that for any positive integer $N$, if $k!\mid N$, then $v_2(k!)\le v_2(N)$.
Thus, $k_\lambda(n)\le k_0:=\max\{k:v_2(k!)\le v_2(\lambda(n))\}$, so it will suffice to show that $k_0!\mid\lambda(n)$
almost surely.

Note that for any positive integer $N$ we have $v_2(N!)=N+O(\log N)$, and
for any prime $p$, $v_p(N!)\le N/(p-1)$.  Using \eqref{eq:v2lambda} we may assume for $n\le x$ that
\begin{equation}
\label{eq:k0}
k_0=\frac{\log_3x}{\log2}+O(\log_4x).
\end{equation}
For $p\ge3$, we have
\[
p^{v_p(k_0!)}\le \exp\left(\frac{k_0\log p}{p-1}\right)\le\exp\left(\frac{k_0\log3}{2}\right)=\exp\left(\frac{\log3}{2\log2}\log_3x+O(\log_4x)\right).
\]
It follows that we may assume for each prime $p\ge3$ that $p^{v_p(k_0!)}\le\exp(0.8\log_3x)=(\log_2x)^{0.8}$.

For $q$ a prime power at most $(\log_2x)^{0.8}$, the number of $n\le x$ not divisible by a prime $r\equiv1\pmod q$
is, by Lemmas \ref{lem:sieve}, \ref{lem:reciprocals}, at most $x/\exp((\log_2x)^{0.19})$.
Summing this count for prime powers up to $(\log_2{x})^{0.8}$ we obtain an expression that is $o(x)$ as $x\to\infty$, so it follows
that but for $o(x)$ choices of $n\le x$ we have $p^{v_p(k_0!)}\mid\lambda(n)$ for all primes $3\le p\le k_0$.
Since by definition we have $2^{v_2(k_0!)}\mid\lambda(n)$, we have $k_0!\mid \lambda(n)$.  This completes
the proof of (iii).
\end{proof}

\begin{thm}
\label{thm:lambdafact}
For a set of integers $n$ of asymptotic density $1$ we have $k_\lambda(n)=\log_3n/\log 2+O(\log_4n)$.
\end{thm}
\begin{proof}
This follows immediately from Lemma \ref{lem:2power}, the definition of $k_0$ in its proof, and \eqref{eq:k0}.
\end{proof}

\begin{rmk}
Let $s_2(m)$ denote the number of 1's in the binary expansion of $m$.
One can show that on a set of asymptotic density 1, $k_\lambda(n)$ is $m+O(\xi(n))$ where $m$ is
the largest integer with $m-s_2(m)\le v_2(\lambda(n))$.
In fact, $m-s_2(m)=v_2(m!)$, so the assertion follows from Lemma~\ref{lem:2power}.

If $k$ is even, then $v_2(k!)=v_2((k+1)!)<v_2(j!)$ for all $j\ge k+2$.  Thus, Lemma \ref{lem:2power} implies
that on a set of asymptotic density 1, $k_\lambda(n)$ is an odd integer.

Let $k_\phi(n)$ denote the largest integer $k$ with $k!\mid\phi(n)$, namely the subject of \S\ref{sec:factorial}.
It follows from the arguments there that if $p$ is the largest prime with $p\#$ (the primorial of $p$)
dividing $\phi(n)$, then on a set of $n$ of asymptotic density 1, $p!\mid\phi(n)$.  In fact, on a set of asymptotic density 1, $p!M\mid\phi(n)$,
 where $M$ is the product of
all of the composite numbers in $(p,\frac32p)$.  (To see this assume $n$ is large and let
$q$ run over the primes to $p$.  If $q>\frac34p$, then $q\nmid M$.  If $\frac17p<q<\frac34p$, then by the method
of \S\ref{sec:factorial}, we may assume that $q^{11}\mid \phi(n)$, so that $v_q(\phi(n))\ge v_q(p!M)$.
In addition, the method of \S\ref{sec:factorial} can also be used to show we may assume that $v_q(\phi(n))>(\log_2x)/(2(q-1))>v_q(p!M)$ for all $q<p/7$.)
Now, by a somewhat stronger version of Bertrand's postulate, we may assume the next prime $r$ after $p$ is $<\frac32p$.
We conclude that $(r-1)!\mid \phi(n)$ and $k_\phi(n)=r-1$.  So on
a set of asymptotic density 1, $k_\phi(n)$ is even.  This is a striking incongruence from the situation with
$k_\lambda(n)$.
\end{rmk}

\section{A related problem}\label{sec:other}
Let ${\mathbb V}=\phi({\mathbb N})$, that is, ${\mathbb V}$ is the set of distinct values of $\phi$.  Let $V(x)$ denote the number of members
of ${\mathbb V}$ in $[1,x]$.  After earlier work of Pillai, Erd\H os, Hall, Maier, and Pomerance, we finally learned the order
of magnitude of $V(x)$ in Ford \cite{ford98}.
 Ignoring subsets of ${\mathbb V}\cap[1,x]$ of size $o(V(x))$ as $x\to\infty$,
what can be said about the largest primorial (or factorial) which divides most members of ${\mathbb V}\cap[1,x]$?  We know
that most values of $\phi$ come from small fibers, and in particular there is a set of integers $S$ of asymptotic density 0
such that $V(x)\sim \#(\phi(S)\cap[1,x])$ as $x\to\infty$.  It seems likely to us that the key function here is exponentially
smaller than $\log_2 x/\log_3x$ and is of the form $(\log_3 x)^{1+o(1)}$.  It would be nice to prove this
assertion.

The analogous problem for Carmichael's function $\lambda$ is even more murky.  Let $V_\lambda(x)$ denote
the number of $\lambda$ values in $[1,x]$.  We do not know the order of magnitude of $V_\lambda(x)$, only
recently learning in \cite{flp} that $V_\lambda(x)=x/(\log x)^{\eta+o(1)}$ as $x\to\infty$, where
$\eta=1-(1+\log\log 2)/\log 2$.

\providecommand{\bysame}{\leavevmode\hbox to3em{\hrulefill}\thinspace}
\providecommand{\MR}{\relax\ifhmode\unskip\space\fi MR }
% \MRhref is called by the amsart/book/proc definition of \MR.
\providecommand{\MRhref}[2]{%
  \href{http://www.ams.org/mathscinet-getitem?mr=#1}{#2}
}
\providecommand{\href}[2]{#2}

\end{document}